\documentclass[12pt]{amsart}
\usepackage{latexsym,amsthm,amssymb,amsfonts}
\usepackage{hyperref}

\newtheorem{theorem}{Theorem}
\newtheorem{lemma}[theorem]{Lemma}
\newtheorem{proposition}[theorem]{Proposition}

\newtheorem*{theoremA}{Theorem A}
\newtheorem*{theoremB}{Theorem B}
\newtheorem*{theoremC}{Theorem C}

\theoremstyle{definition}
\newtheorem{definition}{Definition}

\theoremstyle{remark}

\newcommand{\PGL}{{\mathrm {PGL}}}
\newcommand{\SL}{{\mathrm {SL}}}
\newcommand{\PSL}{{\mathrm {PSL}}}

\newcommand{\SU}{{\mathrm {SU}}}
\newcommand{\PSU}{{\mathrm {PSU}}}

\newcommand{\PO}{{\mathrm {P}\Omega}}

\newcommand{\Sp}{{\mathrm {Sp}}}
\newcommand{\PSp}{{\mathrm {PSp}}}

\newcommand{\Aut}{{\mathrm {Aut}}}
\newcommand{\Out}{{\mathrm {Out}}}

\newcommand{\Irr}{{\mathrm {Irr}}}

\newcommand{\St}{{\mathrm {St}}}

\newcommand{\RR}{{\mathbb R}}
\newcommand{\QQ}{{\mathbb Q}}

\newcommand{\FF}{{\mathbb F}}

\newcommand{\ta}{\hspace{0.5mm}^{2}\hspace*{-0.2mm}}
\newcommand{\tb}{\hspace{0.5mm}^{3}\hspace*{-0.2mm}}

\newcommand{\bZ}{\mathbf{Z}}

\newcommand{\bF}{\mathbf{F}}

\newcommand{\Al}{\textup{\textsf{A}}}
\newcommand{\Sy}{\textup{\textsf{S}}}
\newcommand{\rat}{{\mathrm {rat}}}
\newcommand{\Oinfty}{{\mathcal{O}_\infty}}

\begin{document}

\title[Character degree ratios and composition factors]
{Controlling composition factors of a finite group by its character
degree ratio}

\author{James P. Cossey}
\address{Department of Mathematics, The University of Akron, Akron,
Ohio 44325, USA} \email{cossey@uakron.edu}

\author{Hung Ngoc Nguyen}
\address{Department of Mathematics, The University of Akron, Akron,
Ohio 44325, USA} \email{hungnguyen@uakron.edu}



\date{\today}

\begin{abstract} For a finite nonabelian group $G$ let $\rat(G)$ be the largest ratio of degrees of two
nonlinear irreducible characters of $G$. We show that nonabelian
composition factors of $G$ are controlled by $\rat(G)$ in some
sense. Specifically, if $S$ different from the simple linear groups
$\PSL_2(q)$ is a nonabelian composition factor of $G$, then the
order of $S$ and the number of composition factors of $G$ isomorphic
to $S$ are both bounded in terms of $\rat(G)$. Furthermore, when the
groups $\PSL_2(q)$ are not composition factors of $G$, we prove that
$|G:\Oinfty(G)|\leq \rat(G)^{21}$ where $\Oinfty(G)$ denotes the
solvable radical of $G$.
\end{abstract}

\maketitle


\section{Introduction}

For a finite group $G$, let $b(G)$ denote the maximum degree of an
(ordinary) irreducible character of $G$. If $G$ is nonabelian, we
denote by $c(G)$ the minimum degree of a nonlinear irreducible
character of $G$, and write
\[\rat(G):=\frac{b(G)}{c(G)}.\]
This ratio is referred to as the \emph{character degree ratio} of
$G$. Indeed, $\rat(G)$ is the largest ratio of degrees of two
arbitrary nonlinear irreducible characters of $G$. When $G$ is
abelian, we adopt a convention that $\rat(G)=1$.

The character degree ratio was first introduced and studied by
I.\,M.~Isaacs in~\cite{Isaacs2} in connection with character kernels
of finite groups. It is clear that $\rat(G)=1$ if and only if $G$
has at most two character degrees and the structure of such a group
is fairly simple. It was proved in~\cite[Theorem~12.5]{Isaacs1} that
if $G$ has exactly two character degrees, say $1$ and $m>1$, then
$G$ either has an abelian normal subgroup of index $m$ or is the
direct product of a $p$-group and an abelian group. Consequently, a
group with exactly two character degrees must be metabelian but not
abelian. From this initial result, it is natural to ask whether one
can get some structural information about a finite group from its
character degree ratio.

For instance, Isaacs proved in~\cite[Theorem~6.1 and
Corollary~6.5]{Isaacs2} that the derived length and the Fitting
height of a solvable group $G$ are bounded above respectively by
$3+4\log_2(\rat(G))$ and $3+2\log_2(\rat(G))$. This result indicates
that the derived length and the Fitting height of a solvable group
are controlled by its character degree ratio.

The main goal of this paper is to show that the nonabelian
composition factors of an arbitrary finite group $G$ are somehow
also controlled by $\rat(G)$. Our first result is the following.

\begin{theoremA}\label{main theoremA} Let $G$ be a finite group and $S$ a nonabelian composition factor of
$G$ different from the simple groups $\PSL_2(q)$ with $q$ a prime
power. Then the order of $S$ and the number of times that $S$ occurs
as a composition factor of $G$ are both bounded in terms of
$\rat(G)$.
\end{theoremA}

Let $S$ be a nonabelian composition factor of $G$. By $|S|$ is
bounded in terms of $\rat(G)$ we mean that there is an (increasing)
function $f: \QQ^+\rightarrow\RR$ such that $|S|<f(\rat(G))$. The
family of simple groups $\PSL_2(q)$ is special in this context. It
is known (see~\cite{White} for instance) that the set of character
degrees of $\PGL_2(q)$ is $\{1,q-1,q,q+1\}$ and thus
$\rat(\PGL_2(q))=(q+1)/(q-1)\rightarrow 1$ as  $q\rightarrow
\infty$. Therefore, the exclusion of $\PSL_2(q)$ in Theorem~A is
necessary.

When $G$ is a solvable group, we have seen that the derived length
of $G$ is $\rat(G)$-bounded. However, the order $|G|$ or even the
index $|G:\bF(G)|$ of the Fitting subgroup $\bF(G)$ of $G$ is not
$\rat(G)$-bounded, see Section~\ref{section example} for some
examples. We prove that if the simple groups $\PSL_2(q)$ are not
composition factors of $G$, then the order of the `non-solvable
part' of $G$ is polynomially bounded in terms of $\rat(G)$.

\begin{theoremB}\label{main theoremB} Let $G$ be a finite group with no composition factor
isomorphic to $\PSL_2(q)$ with $q$ a prime power. Let $\Oinfty(G)$
denote the solvable radical of $G$. Then
$|G:\Oinfty(G)|\leq\rat(G)^{21}$.
\end{theoremB}

Theorem~A suggests that when $\rat(G)$ is small enough, then $G$
does not have any nonabelian composition factor different from
$\PSL_2(q)$. We give an upper bound for $\rat(G)$ so that every
nonabelian composition factor $S\neq \PSL_2(q)$ of $G$ disappears.

\begin{theoremC}\label{main theoremC} Let $G$ be a finite group with
$\rat(G)<16/5$. Then the only possible non-abelian composition
factors of $G$ are $\PSL_2(q)$ where $q\geq5$ is a prime power.
Consequently, if the simple groups $\PSL_2(q)$ are not composition
factors of $G$, then $G$ is solvable.
\end{theoremC}

Remark that the bound in Theorem~C can not be improved as
$\rat(\PSL_3(4))=16/5$.

The organization of the paper is as follows. In the next section, we
state a result on character degree ratio of simple groups, which
will be crucial in the proofs of the main theorems. The proof of
this result is presented in Section~\ref{sectionalternatinggroups}
for the alternating groups and in Section~\ref{sectiongroups of Lie
type} for the simple groups of Lie type and simple sporadic groups.
Theorems~A, B, and~C are proved respectively in
Sections~\ref{section proof results A},~\ref{section proof results
B}, and~\ref{section proof theoremC}. In the final section, we
present some examples of solvable groups showing that $|G:\bF(G)|$
is not bounded in terms of $\rat(G)$.


\section{Composition factors and character degree ratio - Theorem~A}\label{section proof results A}

We will prove Theorem~A in this section. To do that, we assume the
following important result on the character degree ratios of
nonabelian simple groups and we postpone its proof until
Sections~\ref{sectionalternatinggroups} and~\ref{sectiongroups of
Lie type}.

\begin{theorem}\label{theorem-extending-character} Let $S$ be a nonabelian simple
group different from $\PSL_2(q)$. Then $S$ has two non-principal
irreducible characters $\alpha$ and $\beta$ extendible to $\Aut(S)$
such that
\[\frac{\alpha(1)}{\beta(1)}>|S|^{1/14}.\]
\end{theorem}

\noindent This theorem will be proved in
Section~\ref{sectionalternatinggroups} for the alternating groups
and in Section~\ref{sectiongroups of Lie type} for simple groups of
Lie type and simple sporadic groups.

Now we observe an obvious but useful lemma.

\begin{lemma}\label{obvious lemma} If $M$ is a normal subgroup of $G$, then $\rat(G/M)\leq \rat(G)$.
\end{lemma}

\begin{proof} This is obvious since every irreducible character of
$G/M$ can be viewed as an irreducible character of $G$.
\end{proof}

The following is an important step towards the proofs of Theorems~A
and~B.

\begin{proposition}\label{proposition solvable radical} Let $N$ be a
nonabelian minimal normal subgroup of $G$ such that the simple
groups $\PSL_2(q)$ are not direct factors of $N$. Then
\[\rat(G)^{14}\geq \rat(G/N)^{14}\cdot|N|.\]
\end{proposition}

\begin{proof} Since $N$ is a
nonabelian minimal normal subgroup of $G$, we see that
$N=S\times \cdots\times S$, a direct product of $k$ copies of a
nonabelian simple group $S$. By
Theorem~\ref{theorem-extending-character}, there exist non-principal
characters $\alpha,\beta\in\Irr(S)$ such that $\alpha$ and $\beta$
extend to $\Aut(S)$ and $\alpha(1)/\beta(1)>|S|^{1/14}$.
By~\cite[Lemma~5]{Bianchi-Lewis}, both the product characters
$\beta\times\cdots\times \beta$ and $\alpha\times\cdots\times\alpha$
extend to $G$. Let $\psi$ and $\chi$ be their extensions,
respectively. First we have
\[\psi(1)=\beta(1)^k \text{ is a nontrivial character degree of } G.\]

Next, since $\chi|_{N}=\alpha\times\cdots\times\alpha$, Gallagher's
Theorem~\cite[Corollary~6.17]{Isaacs1} guarantees that there is a
bijection $\lambda\mapsto \lambda\chi$ between $\Irr(G/N)$ and the
set of irreducible characters of $G$ lying above
$\alpha\times\cdots\times\alpha$. In particular, by taking $\lambda$
to be an irreducible character of $G/N$ of the largest degree, we
deduce that
\[b(G/N)\chi(1)=b(G/N)\alpha(1)^k \text{ is a character degree of } G.\]
It follows that
\[\rat(G)\geq
\frac{b(G/N)\alpha(1)^k}{\beta(1)^k}\geq\rat(G/N)\left(\frac{\alpha(1)}{\beta(1)}\right)^k.\]
Since $\alpha(1)/\beta(1)>|S|^{1/14}$, we obtain \[\rat(G)\geq
\rat(G/N)|S|^{k/14}=\rat(G/N)|N|^{1/14},\] as the proposition
stated.
\end{proof}

We are now ready to prove the first main result.

\begin{proof}[Proof of Theorem~A] Let $S$ be a simple group and $N=S\times\cdots\times S$ a minimal normal subgroup of $G$. Lemma~\ref{obvious lemma} and
Proposition~\ref{proposition solvable radical} imply that
$\rat(G)\geq \rat(G/N)$ if $S$ is abelian or isomorphic to
$\PSL_2(q)$ (where $q$ is a prime power) and $\rat(G)\geq
\rat(G/N)|N|^{1/14}$ otherwise. Applying this repeatedly to
\[N_{i+1}/N_i\lhd G/N_i\] for $i=0,1,...,n-1$ where
\[1=N_0<N_1<\cdots<N_n=G\] is a chief series of $G$, we obtain that
the product of the orders of all nonabelian chief factors which are
not direct products of $\PSL_2(q)$ is at most $\rat(G)^{14}$. We
deduce that, if $S\neq \PSL_2(q)$ is a non-abelian composition
factor of $G$, then the order of $S$ and the number of times that
$S$ occurs as a composition factor of $G$ are both bounded in terms
of $\rat(G)$.
\end{proof}


\section{Bounding the index of the solvable radical - Theorem~B}\label{section proof results B}

We record a result of Mar\'{o}ti on bounding the order of a
particular permutation group.

\begin{lemma}[Mar\'{o}ti~\cite{Maroti}]\label{lemmaMaroti} Let $G$ be a permutation group of degree $n$, and let $d$ be an integer
not less than $4$. If $G$ has no composition factors isomorphic to
an alternating group of degree greater than $d$, then $|G|\leq
d!^{(n-1)/(d-1)}$.
\end{lemma}

This lemma can be used to prove a significant improvement of a
result of Gluck in~\cite[Proposition~2.5]{Gluck}.

\begin{proposition}\label{lemma Gluck} Let $N$ be a normal subgroup of $G$
such that $G/N$ is solvable. Then $|G:\Oinfty(G)|\leq |N|^{1.43}$.
\end{proposition}

\begin{proof} Step 1: We may assume that $\Oinfty(G)=1$.

If $\Oinfty(G)\neq 1$, we replace $G$ by $G/\Oinfty(G)$ and $N$ by
$\Oinfty(G)N/\Oinfty(G)$. Then we use induction on $|G|$ and notice
that $\Oinfty(G/\Oinfty(G))=1$.

\medskip

Step 2: We may assume that $N$ is a unique minimal normal subgroup
of $G$.

Since $\Oinfty(G)=1$ and $G/N$ is solvable, it suffices to show that
$N$ is a minimal normal subgroup of $G$. Assume not. Then there
exists a minimal normal subgroup $M$ properly contained in $N$. Let
$K\subseteq G$ such that $K/M:=\Oinfty(G/M)$. Then by induction on
the order of $N$, we have
\[|G:K|=\left|\frac{G}{M}:\Oinfty\left(\frac{G}{M}\right)\right|\leq
|N/M|^{1.43}\] and \[|K:\Oinfty(K)|\leq |M|^{1.43}.\] Since
$\Oinfty(K)\subseteq \Oinfty(G)$, it follows that
\[|G:\Oinfty(G)|\leq |G:\Oinfty(K)|=|G:K||K:\Oinfty(K)|\leq
|N/M|^{1.43}|M|^{1.43}=|N|^{1.43},\] which proves the proposition.

\medskip

Let $N=S_1\times\cdots\times S_k$, a direct product of copies of a
nonabelian simple group $S$. Then $G$ permutes the subgroups $S_1$,
$S_2$,..., $S_k$ by conjugation. Let $H$ denote the kernel of this
action. Then $H$ has an embedding into $\Aut(S_1)\times\cdots\times
\Aut(S_k)$ and $G/H$ has an embedding into the symmetric group
$\Sy_k$.

\medskip

Step 3: $|G|\leq (\sqrt{24}|\Aut(S)|)^k$.

We note that $N\subseteq H$ and recall that $G/N$ is solvable. Thus
$G/H$ is solvable as well. So $G/H$ is a solvable permutation group
of degree $k$ and therefore
\[|G/H|\leq 24^{(k-1)/3}\leq \sqrt[3]{24}^k\] by
Lemma~\ref{lemmaMaroti}. On the other hand, as $H$ is embedded into
$\Aut(S_1)\times\cdots\times \Aut(S_k)$, we have
\[|H|\leq |\Aut(S)|^k.\]
Now we deduce that
\[|G|=|G/H||H|\leq (\sqrt[3]{24}|\Aut(S)|)^k.\]

\medskip

Step 4: $|G|\leq |N|^{1.43}$.

It suffices to show that $|\Out(S)|\leq |S|^{0.43}/\sqrt[3]{24}$ for
every nonabelian simple group $S$. If $S=\Al_5$, $\PSL_2(7)$, or
$\Al_6$ then we can check the inequality directly. So we may assume
that $|S|\geq 504$. Then $\sqrt[3]{24}\leq |S|^{0.171}$ and hence it
suffices to show that $|\Out(S)|\leq |S|^{0.259}$ for every
nonabelian simple group $S$. But this is done by a routine check on
the list of simple groups together with their outer automorphism
groups, see~\cite{Gorenstein} for instance.
\end{proof}

Now we can prove Theorem~B.

\begin{proof}[Proof of Theorem~B] We argue by induction on the order of $G$. Since $\rat(G/\Oinfty(G))\leq
\rat(G)$ by Lemma~\ref{obvious lemma}, we may assume that
$\Oinfty(G)=1$. We want to show that $|G|\leq \rat(G)^{21}$.

If $G$ is trivial then we are done, so assume that $G$ is
nontrivial. Then $G$ has a nonabelian minimal normal subgroup $N$.
We note that the groups $\PSL_2(q)$ are not the direct factors of
$N$ since they are not the composition factors of $G$. Let $M\unlhd
G$ such that
\[\frac{M}{N}=\Oinfty(G/N).\]
As $\Oinfty(M)$ is characteristic in $M$ and $M\unlhd G$, we see
that $\Oinfty(M)\unlhd G$. Thus $\Oinfty(M)=1$ since $G$ has trivial
solvable radical. We note that $M/N$ is solvable. Therefore,
Proposition~\ref{lemma Gluck} implies that
\[|M|=|M:\Oinfty(M)|\leq |N|^{1.43}.\]

Using induction hypothesis for $G/N$, we have
\[|G:M|=|G/N:\Oinfty(G/N)|\leq \rat(G/N)^{21}.\] It follows that
\[|G|\leq |M|\rat(G/N)^{21}\leq |N|^{1.43}\cdot\rat(G/N)^{21}.\]
Together with Proposition~\ref{proposition solvable radical}, we
deduce that
\[|G|\leq \rat(G)^{21},\] which completes the proof.
\end{proof}



\section{Character degree ratio of the alternating groups}\label{sectionalternatinggroups}

In this section, we prove Theorem \ref{theorem-extending-character}
for the alternating groups with $n \geq 7$ (note that $\Al_5 \cong
\PSL_2(5)$ and $\Al_6 \cong \PSL_2(9)$ are excluded from the
theorem). First, however, we review some of the basic results on
characters of the symmetric and alternating groups, proofs of which
can be found in~\cite{jameskerber} for instance.

For a positive integer $n$, a partition $\lambda$ of $n$ is defined
to be a nonincreasing sequence of positive integers $(\lambda_1,
\lambda_2, \ldots, \lambda_k)$ that sum to $n$.  We adopt the \lq
\lq exponential" notation for repeated terms in a partition, so that
for instance the partition $(5, 4, 4, 4, 3, 3, 1)$ would be written
$(5, 4^3, 3^2, 1)$.  Recall that the ordinary irreducible
representations of the symmetric group $\Sy_n$ are indexed by the
partitions of $n$, and if $\lambda$ is a partition of $n$, we let
$\chi_{\lambda}$ denote the irreducible character of $\Sy_n$
corresponding to $\lambda$. For the node $(i, j)$ in row $i$ and
column $j$ of the Young diagram of $\lambda$, we let $h_{\lambda}(i,
j)$ denote the hook length of the node $(i,j)$, which is computed by
counting the number of nodes of $\lambda$ directly to the right and
directly below $(i, j)$, including $(i, j)$ itself.  We denote by
$H_{\lambda}$ the product of all the hook lengths of $\lambda$, and
we note that the hook length formula shows that \[\chi_{\lambda}(1)
= \frac{n!}{H_{\lambda}}.\]

Recall that the conjugate $\lambda'$ of the partition $\lambda$ is
the partition obtained by interchanging the rows and the columns of
$\lambda$. We say $\lambda$ is self-conjugate if $\lambda' =
\lambda$. If $\lambda$ is a partition of $n$, then $\chi_{\lambda}$
restricts irreducibly to $\Al_n$ if and only if $\chi$ is not
self-conjugate.

For $n \geq 4$, the partition $(n-1, 1)$ is certainly not
self-conjugate, and $\chi_{(n-1, 1)}(1) = n-1$.  Thus to prove
Theorem~\ref{theorem-extending-character} for the alternating
groups, it suffices to find a character $\chi$ of $\Al_n$ such that
$\chi(1) > (n!)^{1/14}(n-1)$. Note that for a non self-conjugate
partition $\lambda$ of $n$, $\chi_{\lambda}(1) > (n!)^{1/14}(n-1)$
if and only if
$$\frac{(n!)^{13/14}}{n-1} > H_{\lambda}.$$ To prove Theorem~\ref{theorem-extending-character} for the alternating groups, then,
it is enough to prove the following:

\begin{proposition}\label{14boundforS_n} Let $n \geq 7$ be an integer. Then there exists a partition $\lambda$ of $n$ that is not self-conjugate such that \[\frac{(n!)^{13/14}}{n-1} >
H_{\lambda}.\] In particular,
Theorem~\ref{theorem-extending-character} is true for the
alternating groups.
\end{proposition}

We begin by proving a lower bound on $(n!)^{13/14}/(n-1)$.

\begin{lemma}\label{lowerbound} For any positive integer $n \geq 15$, we have $$\frac{(n!)^{13/14}}{n-1} > 1.35 \left( \frac{n}{e} \right)^{25n/28}.$$
\end{lemma}

\begin{proof} By Stirling's formula, we have $$(n!)^{13/14} \geq \left (\frac{n}{e} \right)^{13n/14}(2\pi n)^{13/28}.$$
Thus
$$\frac{(n!)^{13/14}}{n-1} > \frac{\left( \frac{n}{e} \right)^{13n/14}(2\pi n)^{13/28}}{n} = \left(\frac{n}{e}\right)^{13n/14 - 15/28} \left( \frac{(2\pi)^{13}}{e^{15}} \right) ^{1/28}.$$

Note that $\left( \frac{(2\pi)^{13}}{e^{15}} \right) ^{1/28} >
1.35$.  Moreover, since $n \geq 15$, we have $\frac{13n}{14} -
\frac{15}{28} \geq \frac{25n}{28}$. Thus $$\frac{(n!)^{13/14}}{n-1}
> 1.35 \left(\frac{n}{e} \right)^{{25n}/{28}},$$ as claimed.
\end{proof}

We will now work to develop non self-conjugate partitions with large
degree, by finding non self-conjugate partitions with \lq \lq small"
$H_{\lambda}$.  Recall that if $\lambda$ and $\mu$ are partitions, we say that $\lambda \leq \mu$ if, for each $i$, $\lambda_i \leq \mu_i$.

\begin{definition} Let $m$ be a positive integer. Define $\Gamma_m$ to be the set of partitions $\lambda$ such that $(m^m) \leq \lambda \leq ((m+2)^m)$.
\end{definition}

Note that if $\lambda \in \Gamma_m$ and $\lambda$ is a partition of
$n$, then $m^2 \leq n \leq m^2 + 2m = (m + 1)^2 - 1$. Thus for each
$n$, there is a unique value of $m$ such that there is a partition
of $n$ contained in $\Gamma_m$.  Moreover, with the exception the
partition $(m^m)$ of $n = m^2$, every partition in $\Gamma_m$ is not
self-conjugate.

\begin{lemma}\label{upperbound} Let $m$ be a positive integer and suppose $\lambda \in \Gamma_m$. Then $H_{\lambda} < (m + 1)^{(m+1)^2}$.
\end{lemma}

\begin{proof} Note that if $\lambda \in \Gamma_m$, then $\lambda$ is properly contained in the partition $\mu_{(m)} = ((m+2)^m)$. Thus $H_{\lambda} < H_{\mu_{(m)}}$.

Observe that the hook lengths in the first column of $\mu_{(m)}$ are
$(2m+1, 2m, \ldots, m + 2)$, and thus by the geometric mean
inequality we have the product of the hook lengths in the first
column of $\mu_{(m)}$ is $$\prod_{i=2}^{m+1} (m + i) \leq \left( \frac{(2m+1) +
(2m) + \cdots + (m+2)}{m} \right) ^{m}  = \left( \frac{3m + 3}{2}
\right)^{m}.$$

Similarly the product of hook lengths in column $j$ of $\mu_{(m)}$, for $1 \leq j
\leq m + 2$, is bounded above by $$\left( \frac{3m + 5 - 2j}{2}
\right)^{m}.$$   Thus $$H_{\lambda} < \left[ \left( \frac{3m + 3}{2}
\right) \cdots\left( \frac{m+1}{2} \right)\right]^{m} \leq (m +
1)^{m^2 + 2m} < (m+1)^{(m+1)^2},$$ where we are again using the
geometric mean inequality, and we are done.
\end{proof}

\begin{lemma} Suppose $n \geq 64$. If $\lambda \in \Gamma_m$ is a partition of $n$ (thus $m \geq 8)$ then $$\frac{(n!)^{13/14}}{n-1} \geq H_{\lambda}.$$
\end{lemma}

\begin{proof} By Lemma \ref{lowerbound} we have $$\frac{(n!)^{13/14}}{n-1} > 1.35 \left( \frac{n}{e}
\right)^{25n/28} > \left( \frac{n}{e} \right)^{25n/28}.$$ By Lemma
\ref{upperbound} we have $H_{\lambda} < (m + 1)^{(m+1)^2}$. Note
that since $m \geq 8$, then we have $\frac{81}{64}n \geq
\frac{81}{64}m^2 \geq (m+1)^2$. Therefore $$H_{\lambda} \leq
(m+1)^{(m+1)^2} \leq \left[\left(\frac{81}{64}
n\right)^{{1}/{2}}\right]^{(m+1)^2} \leq \left[\left(\frac{81}{64}
n\right)^{{1}/{2}}\right]^{{81n}/{64}} =  \left( \frac{81}{64} n
\right)^{{81n}/{128}}.$$

Thus it is enough to show that $\left( \frac{81}{64} n
\right)^{{81n}/{128}} \leq \left( \frac{n}{e} \right)^{25n/28}$.
After taking $n$th roots of both sides, this is equivalent to
showing $\left( \frac{81}{64} n \right)^{{81}/{128}} \leq \left(
\frac{n}{e} \right)^{25/28}$. However, basic algebra shows this
inequality is true for all $n \geq 55$.
\end{proof}

We are now ready to prove Proposition~\ref{14boundforS_n}, which
will prove Theorem~\ref{theorem-extending-character} for $A_n$ with
$n \geq 7$.

\begin{proof}[Proof of Proposition~\ref{14boundforS_n}] Note that the partition $(m^m) \in \Gamma_m$ is self-conjugate,
so we may not use that partition if $n = m^2$.  However, by the
branching rule (see \cite{jameskerber} for instance),
$\chi_{(m^m)}(1) = \chi_{(m^{m-1}, m-1)}(1) \leq \chi_{(m + 1,
m^{m-2}, m-1)}(1)$, and the partition $(m + 1, m^{m-2}, m-1)$ of
$m^2$ is not self-conjugate for $m > 1$.

Thus if $n \geq 64$ (i.e. $m\geq 8$), we are done by the previous
lemma. To finish the proof, we need to exhibit non self-conjugate
partitions for all $n$ such that $7 \leq n \leq 63$ that have large
enough degree.  However, this is checked as follows: for instance,
for all $n$ such that $49 \leq n \leq 63$, there is a non
self-conjugate partition $\lambda$ of $n$ which contains the
partition $(7^7)$ with degree at least $$\chi_{(7^7)}(1) \approx
4.75 \times 10^{23} > 1.07 \times 10^8 \approx (63-1)(63!)^{1/14}
\geq (n-1)(n!)^{1/14}.$$  Similar calculations can be easily done
for partitions in $\Gamma_6$, $\Gamma_5$, $\Gamma_4$, and
$\Gamma_3$. The partitions $(4, 2, 2)$ and $(3, 2, 2)$ of $8$ and
$7$, respectively, complete the proof.
\end{proof}

We mention here that the above method could be extended to prove
better bounds in Theorem~\ref{theorem-extending-character} for the
alternating groups $\Al_n$ in terms of powers of $n!$. However, the
best bound for the other simple groups is on the order of
$|S|^{1/14}$, so we do not feel it necessary to improve our bound
for the alternating groups.

\section{Character degree ratio of the simple groups of Lie type}\label{sectiongroups of Lie type}

We now prove Theorem~\ref{theorem-extending-character} for the
simple groups of Lie type and then finish off the proof of the
theorem. We will choose the larger character $\alpha$ to be the
so-called \emph{Steinberg character} and the smaller character
$\beta$ usually to be the smallest unipotent character. We mainly
follow the notation in~\cite{Carter} for finite groups of Lie type
and their characters, especially unipotent characters.

We make essential use of the Lusztig's classification of unipotent
characters of finite groups of Lie type. We summarize some main
points here very briefly. Let $\mathcal{G}$ be a simple algebraic
group over $\overline{\FF}_q$ and $F$ a Frobenius endomorphism on
$\mathcal{G}$. Set $G:=\mathcal{G}^F$. Let $\mathcal{G}^*$ be the
dual group of $\mathcal{G}$, $F^\ast$ the dual Frobenius
endomorphism, as defined in~\cite[\S13.10]{Digne-Michel} and denote
$G^\ast:={\mathcal{G}^\ast}^{F^\ast}$. Lusztig's classification (see
Chapter 13 of \cite{Digne-Michel}) states that the set of
irreducible complex characters of $G$ is partitioned into Lusztig
series $\mathcal{E}(G,(s))$ associated to various conjugacy classes
$(s)$ of semisimple elements of $G^\ast$. In fact,
$\mathcal{E}(G,(s))$ consists of irreducible constituents of the
Deligne-Lusztig character $R_\mathcal{T}^\mathcal{G}(\theta)$, where
$(\mathcal{T},\theta)$ is of the geometric conjugacy class
associated to $(s)$. The characters in $\mathcal{E}(G,(1))$ are
called \emph{unipotent characters} of $G$.

By the results of Lusztig, unipotent characters of isogenous groups
have the same degrees. These degrees are polynomials in $q$ and they
are essentially products of cyclotomic polynomials evaluated at $q$.
The unipotent characters of the linear and unitary groups are
labeled by partitions while those of the symplectic and orthogonal
groups are labeled by the so-called symbols. For finite classical
groups, the degrees of unipotent characters are given by some
delicate combinatorial formulas. On the other hand, the degrees of
unipotent characters of exceptional groups have been computed
explicitly, see~\cite[\S 13.8]{Carter} for more details.

We need the following result, due to Lusztig, that appears as
Theorem~2.5 of~\cite{Malle}.

\begin{lemma}[Lusztig~\cite{Lusztig}]\label{lemma Malle 1} Let $\chi$ be a unipotent character of a simple group
of Lie type $S$. Then $\chi$ is $\Aut(S)$-invariant, except in the
following cases:
\begin{enumerate}
\item[(i)] In $S=\PO_{2n}^+(q)$ with even $n\geq4$, the graph
automorphism of order $2$ interchanges the two unipotent characters
in all pairs labeled by the same degenerate symbols of defect $0$
and rank $n$.

\item[(ii)] In $S=\PO_8^+(q)$, the graph automorphism of order $3$ has
two orbits with characters labeled by the symbols
$$\left\{{2\choose 2}, {2\choose 2}', {0\hspace{6pt} 1 \choose 1\hspace{6pt} 4}\right\} \text{ and } \left\{{1\hspace{6pt}2\choose 1\hspace{6pt}2}, {1\hspace{6pt}2\choose 1\hspace{6pt}2}',
{0\hspace{6pt} 1 \hspace{6pt} 4\choose
1\hspace{6pt}2\hspace{6pt}4}\right\}.$$

\item[(iii)] In $S=\Sp_4(2^{2f+1})$, the graph
automorphism of order $2$ interchanges the two unipotent characters
labeled by the symbols
$${1\hspace{6pt} 2 \choose 0} \text{ and } {0\hspace{6pt} 1 \choose 2}.$$

\item[(iv)] In $S=G_2(3^{2f+1})$, the graph
automorphism of order $2$ interchanges the two unipotent characters
denoted by $\phi'_{1,3}$ and $\phi''_{1,3}$.

\item[(v)] In $S=F_4(2^{2f+1})$, the graph automorphism of order $2$ has
eight orbits of length $2$ with characters denoted by
$$\{\phi'_{8,3},\phi''_{8,3}\}, \{\phi'_{8,9},\phi''_{8,9}\}, \{\phi'_{2,4},\phi''_{2,4}\}, \{\phi'_{2,16},\phi''_{2,16}\},$$
$$\{\phi'_{9,6},\phi''_{9,6}\}, \{\phi'_{1,12},\phi''_{1,12}\}, \{\phi'_{4,7},\phi''_{4,7}\}, \{B_2,\epsilon'\}, \{B_2,\epsilon''\}.$$
\end{enumerate}
\end{lemma}

Malle has observed in~\cite[Theorem~2.4]{Malle} a very important
property of unipotent characters.

\begin{lemma}[Malle~\cite{Malle}]\label{lemma Malle 2} Let $S$ be a simple group of Lie type. Then every
unipotent character of $S$ extends to its inertial group in
$\Aut(S)$.
\end{lemma}

Combining Lemmas~\ref{lemma Malle 1} and~\ref{lemma Malle 2}, we see
that almost all of the unipotent characters of a simple group of Lie
type are extendible to the full automorphism group. The character
table of the Tits group is available in~\cite{Atl1} and we will
treat it as a sporadic simple group rather than a simple group of
Lie type.

\begin{proposition}\label{proposition groups of lietype} Theorem~\ref{theorem-extending-character} is true for the simple groups of Lie
type.
\end{proposition}

\begin{proof} For each family of simple groups of Lie type $S$, we will
create two unipotent characters that are extendible to $\Aut(S)$ and
their degree ratio is greater than $|S|^{1/14}$. These characters of
course cannot be those singled out in Lemma~\ref{lemma Malle 1}.
Indeed, the character of larger degree will be chosen to be the
Steinberg character of $S$, denoted by $\St_S$. Assuming that $p$ is
the defining characteristic of $S$, then it is well-known that
$\St_S$ is extendible to $\Aut(S)$ and $\St_S(1)=|S|_p$,
see~\cite{Feit}.

\medskip

1) Linear groups $\PSL_n(q)$ and unitary groups $\PSU_n(q)$ in
dimension $n\geq3$: The unipotent characters of $S$ are labeled by
partitions of $n$, see~\cite[p. 465]{Carter}. We consider the
character $\chi^\lambda$ labeled by the particular partition
$\lambda=(n-1, 1)$. This character is extendible to $\Aut(S)$ by
Lemmas~\ref{lemma Malle 1} and~\ref{lemma Malle 2} and has degree
$$\chi^\lambda(1)=\frac{q^n-q}{q-1} \text{ if } S=\PSL_n(q)$$ and $$\chi^\lambda(1)=\frac{q^n+q(-1)^n}{q+1} \text{ if } S=\PSU_n(q).$$
We note that $\St_S(1)=q^{n(n-1)/2}$ and it is easy to check that
$\St_S(1)/\chi^\lambda(1)>|S|^{1/14}$ when $n=3$. So let us assume
that $n\geq 4$. We have
\[\frac{\St_S(1)}{\chi^\lambda(1)}\geq \frac{q^{n(n-1)/2}}{(q^n-q)/(q-1)}>\frac{q^{n(n-1)/2}}{q^n}=q^{n(n-3)/2}.\]
On the other hand, we see that $|\PSL_n(q)|<|\PSU_n(q)|<q^{n^2-1}$.
Since it is clear that $q^{n(n-3)/2}>q^{(n^2-1)/14}$ for every
$n\geq4$, we conclude that
\[\frac{\St_S(1)}{\chi^\lambda(1)}>|S|^{1/14}.\]

\medskip

2) Symplectic groups $\PSp_{2n}(q)$ and orthogonal groups
$\Omega_{2n+1}(q)$ with $n\geq2$: According to \cite[p.
466]{Carter}, the unipotent characters of these groups are labeled
by symbols of the form
$${\lambda \choose \mu}={\lambda_1\hspace{6pt}\lambda_2\hspace{6pt}\lambda_3\hspace{6pt}
...\hspace{6pt}\lambda_a\choose\mu_1\hspace{6pt}\mu_2\hspace{6pt}...\hspace{6pt}\mu_b}$$
where $0\leq\lambda_1<\lambda_2<... < \lambda_a$,
$0\leq\mu_1<\mu_2<...<\mu_b$, $a-b$ is odd and positive,
$(\lambda_1, \mu_1)\neq (0,0)$, and
$$\sum_i\lambda_i+\sum_j\mu_j-\left(\frac{a+b-1}{2}\right)^2=n.$$ Now the character labeled by the symbol ${\lambda\choose\mu}={0\hspace{6pt} 1 \hspace{6pt} n
\choose-}$ is extendible to $\Aut(S)$ and its degree is
$$\chi^{\lambda\choose\mu}(1)=\frac{(q^n-1)(q^n-q)}{2(q+1)}<q^{2n}.$$
We note that $\St_S(1)=q^{n^2}$ and it follows that
\[\frac{\St_S(1)}{\chi^{\lambda\choose\mu}(1)}>\frac{q^{n^2}}{q^{2n}}=q^{n^2-2n}.\]
As $|\PSp_{2n}(q)|=|\Omega_{2n+1}(q)|\leq q^{2n^2+n}$, the
proposition follows in this case.

\medskip

3) Orthogonal groups $\mathrm{P}\Omega^+_{2n}(q)$ with $2n\geq 8$:
By~\cite[p. 471]{Carter}, the unipotent characters of these groups
are labeled by symbols of the same form as the symplectic groups
except that $a-b\equiv0(\bmod$ $4)$ and
$$\sum_i\lambda_i+\sum_j\mu_j-\left[\left(\frac{a+b-1}{2}\right)^2\right]=n.$$ We consider the character labeled by the symbol ${\lambda\choose\mu}={n-1 \choose1}$. This character is
$\Aut(S)$-invariant by Lemma~\ref{lemma Malle 1} and therefore
extendible to $\Aut(S)$ by Lemma~\ref{lemma Malle 2}. Its degree is
$$\chi^{\lambda\choose\mu}(1)=\frac{(q^n-1)(q^n+q)}{q^2-1}<q^{2n}.$$ Since
$\St_S(1)=q^{n(n-1)}$ and $|\mathrm{P}\Omega^+_{2n}(q)|<
q^{2n^2-n}$, we are done in this case.

\begin{table}[h]
\caption{Unipotent characters of simple exceptional groups of Lie
type.}\label{table}
\begin{tabular}{lll}
\hline
 Group & Characters  & Degrees\\ \hline

$\ta B_2(q)$&$\ta B_2[a]$  & $(q-1)\sqrt{q/2}$\\
&$\St$  & $q^{2}$\medskip\\

 $\tb D_4(q)$&$\phi'_{1,3}$  & $q\Phi_{12}$\\
&$\St$  & $q^{12}$\medskip\\

$G_2(q)$&$\phi_{2,1}$  & $q\Phi_{2}^2\Phi_3/6$\\
&$\St$  & $q^{6}$\medskip\\

$\ta G_2(q)$&cuspidal 1  & $(q^2-1)\sqrt{q/3}$\\
&$\St$  & $q^{3}$\medskip\\

 $F_4(q)$&$\phi_{4,1}$  & $q\Phi_{2}^2\Phi_6^2\Phi_8/2$\\
&$\St$  & $q^{24}$\medskip\\

 $\ta F_4(q)$&$\epsilon'$  & $q\Phi_{6}\Phi_{12}$\\
&$\St$  & $q^{12}$\medskip\\

 $E_6(q)$&$\phi_{6,1}$  & $q\Phi_{8}\Phi_{9}$\\
&$\St$  & $q^{36}$\medskip\\

 $\ta E_6(q)$&$\phi'_{2,4}$  & $q\Phi_{8}\Phi_{18}$\\
&$\St$  & $q^{36}$\medskip\\

$E_7(q)$&$\phi_{7,1}$  & $q\Phi_{7}\Phi_{12}\Phi_{14}$\\
&$\St$  & $q^{63}$\medskip\\

$E_8(q)$&$\phi_{8,1}$  & $q\Phi_4^2\Phi_{8}\Phi_{12}\Phi_{20}\Phi_{24}$\\
&$\St$  & $q^{120}$\medskip\\
\hline
\end{tabular}
\end{table}

4) Orthogonal groups $\mathrm{P}\Omega^-_{2n}(q)$ with $2n\geq 8$:
From~\cite[p. 475]{Carter}, we know that the unipotent characters of
these groups are labeled by symbols of the same form as the
orthogonal groups of plus type in even dimension except that $a>b$
and $a-b\equiv2(\bmod$ $4)$. We consider the character labeled by
the symbol ${\lambda\choose\mu}={1\hspace{6pt}n-1 \choose-}$. This
character has degree
$$\chi^{\lambda\choose\mu}(1)=\frac{(q^n+1)(q^n-q)}{q^2-1}<q^{2n}$$ and we
are also done since $|\mathrm{P}\Omega^-_{2n}(q)|<q^{2n^2-n}$.

\medskip

5) Exceptional groups: Since all unipotent characters of exceptional
groups are explicitly computed, it is easy to find two unipotent
characters that we need. Table~\ref{table} displays these characters
as required. In the table, we denote by $\Phi_k:=\Phi_k(q)$ the
$k$th cyclotomic polynomial evaluated at $q$,
see~\cite[\S13.9]{Carter}. We choose the bounding constant $1/14$
since for the Ree groups $\ta G_2(q)$, the ratio
$q^3/(q^2-1)\sqrt{q/3}$ is approximately equal to $\sqrt{3q}$ and
$|\ta G_2(q)|$ is approximately equal to $q^7$.
\end{proof}

Now we can complete the proof of
Theorem~\ref{theorem-extending-character}.

\begin{proof}[Proof of Theorem~\ref{theorem-extending-character}] The simple groups of Lie type have
been handled above. The alternating groups are treated in
Section~\ref{sectionalternatinggroups}. Finally, the sporadic simple
groups and the Tits group can be checked directly by
using~\cite{Atl1}.
\end{proof}


\section{Character degree ratio and solvability - Theorem~C}\label{section proof theoremC}

This section is devoted to the proof of Theorem~C. We start with an
analogue of Theorem~\ref{theorem-extending-character}.

\begin{lemma}\label{lemma groups of Lie type 16/5} Let $S$ be a simple group of Lie type different from $\PSL_2(q)$ for any prime power
$q$. Then there exist non-principal irreducible characters $\alpha$
and $\beta$ of $S$ so that both $\alpha$ and $\beta$ extend to
$\Aut(S)$ and \[\frac{\alpha(1)}{\beta(1)}\geq\frac{16}{5}.\]
\end{lemma}

\begin{proof} We will essentially use the same characters as in the proof of Theorem \ref{main theoremC}.
For the linear groups $\PSL_n(q)$ with $n \geq 4$, we have
$$\rat(S) \geq \frac{q^{{n(n-1)}/{2}}}{{(q^n - q)}{(q - 1)}} >
q^{\frac{n(n-3)}{2}} \geq q^2 > \frac{16}{5}.$$ For $\PSL_3(q)$, we
still have the same inequality as long as $q\geq 4$. For $\PSL_3(3)$
we use the characters of degrees $39$ and $12$ instead and note that
$39/12>16/5$. We remark that $\PSL_3(2)\cong\PSL_2(7)$ is excluded
from our consideration.

Similarly for the unitary groups $\PSU_n(q)$ with $n\geq4$, we have
$$\rat(S)
\geq \frac{q^{{n(n-1)}/{2}}}{(q^n+q(-1)^n)/(q+1)} \geq q^2 > 16/5.$$
On the other hand, for $\PSU_3(q)$ with $q\neq 2$, we have
$$\rat(S)\geq \frac{q^3}{q^2-q}>\frac{16}{5}.$$

For the symplectic groups $\PSp_{2n}(q)$ and orthogonal groups
$\Omega_{2n + 1}(q)$ with $n\geq2$, we have $$\rat(S) \geq
\frac{q^{n^2}}{{(q^n - 1)(q^n - q)}/{2(q + 1)}} > 2q^{n^2 - 2n + 1}
\geq 4.$$ For the orthogonal groups $\mathrm{P}\Omega^\pm_{2n}(q)$
with $n\geq 4$, we have $$\rat(S) > \frac{q^{n(n-1)}}{q^{2n}}\geq
q^{4} > 16/5.$$

For the exceptional groups of Lie type, the characters in
Table~\ref{table} show that $\rat(G) > 16/5$. This is easily checked
for the \lq \lq small" values of $q$. For the larger values, this
can be seen by recalling that $\Phi_n$ has degree
$\phi(n)$\footnote{Here $\phi(n)$ is the Euler's totient function
that counts the number of positive integers less than or equal to n
that are relatively prime to n.} and, for $n < 105$, the
coefficients of $\Phi_n$ are $\pm 1$. Thus $\Phi_n(q) \leq (\phi(n)
+ 1) q^{\phi(n)}$, which allows one to obtain the bounds $\rat(G)
> 16/5$.

The sporadic simple groups may be easily checked using~\cite{Atl1}.
Finally, for the alternating groups, it is enough to show that there
exists a character of $\Al_n$ of degree at least $16(n-1)/5$.
However, we have already shown that there exists a character degree
of $\Al_n$ at least $(n!)^{1/14} (n-1)$, and for $n \geq 11$, we
have $(n!)^{1/14}
> 16/5$.  The result can be checked for $7 \leq n \leq 10$ using
the character tables given in~\cite{Atl1,jameskerber}, and we are
done.
\end{proof}

\begin{proof}[Proof of Theorem~C] By applying Lemma~\ref{lemma groups of Lie type 16/5}, the proof goes along the same lines
as that of Theorem~A. We leave details to the reader.
\end{proof}


\section{Examples}\label{section example} In this section we give examples to show that the index of the
Fitting subgroup need not be bounded in terms of $\rat(G)$, even if
$G$ is solvable. This is in contrast to the situation with character
degrees, where bounds on $|G : \bF(G)|$ are known in terms of the
largest character degree $b(G)$. For instance, it is shown by
Gluck~\cite{Gluck} that there is a constant $L$ such that
$|G:\bF(G)|\leq b(G)^L$ for every finite group $G$. When $G$ is
solvable, Moret\'{o} and Wolf~\cite{Moreto-Wolf} proved that
$|G:\bF(G)|\leq b(G)^3$.

Recall that $\rat(G)$ is defined to be $1$ if $G$ has at most one
nonlinear character degree. We first give an example of an infinite
family of solvable groups, each with exactly one nonlinear character
degree, such that $|G : \bF(G)|$ can be arbitrarily large.

Let $m$ be any positive integer with $m > 1$, and let $p$ be a prime
of the form $md + 1$ for some positive integer $d$ (note there are
infinitely many such primes by Dirichlet's theorem). Let $X$ be the
Frobenius group formed by the action of $\FF_p^\star$ on $F:=
\FF_p$. Let $G$ be the subgroup of $X$ containing $F$ such that $|G
: F| = m$. Then the character degrees of $G$ are precisely $1$ and
$m$, although the character degree $m$ could have large
multiplicity. Thus $\rat(G) = 1$ in this case.  As $m$ goes to
infinity, the index $|G : F|$ is arbitrarily large, thus showing
that $|G : \bF(G)|$ is arbitrarily large.

One might want an example in which $G$ has more than one nonlinear
character degree, which we now provide. To do so, we examine a
specific case of a more general example constructed in Section 3 of
\cite{noritzsch}, and more details can be found in that paper.  Let
$m = p^i$ for some prime $p$ and some positive integer $i$, and let
$n = p^i + 1$.  Let $S$ be a Sylow $p$-subgroup of $\SU_3(p^{2i})$.
Let $N$ be a maximal subgroup of $\bZ(S)$ and define $P = S/N$. Note
that $P$ is extraspecial.  Let $\lambda$ be an element of order $p^i
+ 1$ of the multiplicative group of the field of order $p^{2i}$, and
define $H$ to be the group generated by $\lambda$. Then $H$ acts
Frobeniusly on $P/\bZ(P)$ and centralizes $\bZ(P)$. Let $G$ be the
semidirect product of $H$ acting on $P$. Note that $G$ is solvable,
and it can be shown that the character degrees of $G$ are $1$,
$p^i$, and $p^i + 1$ and that $|G : \bF(G)| = p^i + 1$. Thus as $p$
or $i$ goes to infinity, we see that $\rat(G)$ goes to 1 while $|G :
\bF(G)|$ goes to infinity.


\end{document}